\documentclass[12pt,reqno]{article}

\usepackage[usenames]{color}
\usepackage{amssymb}
\usepackage{amsmath}
\usepackage{amsthm}
\usepackage{amsfonts}
\usepackage{amscd}
\usepackage{graphicx}

\usepackage[colorlinks=true,
linkcolor=webgreen,
filecolor=webbrown,
citecolor=webgreen]{hyperref}

\definecolor{webgreen}{rgb}{0,.5,0}
\definecolor{webbrown}{rgb}{.6,0,0}

\usepackage{color}
\usepackage{fullpage}
\usepackage{float}

\usepackage{graphics}
\usepackage{latexsym}

\begin{document}

\theoremstyle{plain}
\newtheorem{theorem}{Theorem}
\newtheorem{corollary}[theorem]{Corollary}
\newtheorem{proposition}{Proposition}
\newtheorem{lemma}{Lemma}
\newtheorem{example}{Examples}
\newtheorem{remark}{Remark}

\begin{center}
{\large\bf
Harmonic Catalan Number Series \\[8pt]
via Half-Integer Binomial Coefficients
}

\vskip 1cm

{\large
Kunle Adegoke \\
Department of Physics and Engineering Physics \\ Obafemi Awolowo University, 220005 Ile-Ife, Nigeria \\
\href{mailto:adegoke00@gmail.com}{\tt adegoke00@gmail.com}

\vskip 0.2 in

Robert Frontczak \\
Independent Researcher\\
 72764 Reutlingen, Germany \\
\href{mailto:robert.frontczak@web.de}{\tt robert.frontczak@web.de}
 
\vskip 0.2 in

Taras Goy \\
Faculty of Mathematics and Computer Science \\
Vasyl Stefanyk Carpathian National University \\
76018 Ivano-Frankivsk, Ukraine \\
\href{mailto:taras.goy@pnu.edu.ua}{\tt taras.goy@pnu.edu.ua}
}

\end{center}

\vskip .15 in

\begin{abstract}
We develop a method for evaluating finite and infinite series involving Catalan numbers by specializing alternating binomial identities 
at half-integer parameters. The key ingredient is an explicit representation of the generalized binomial coefficients $\binom{m+\frac12}{k+1}$
in terms of Catalan numbers and products of odd linear factors. Combining this representation with Pascal-type identities and two identities of Bat{\i}r and Sofo, we derive several families of closed-form summation formulas involving Catalan numbers, harmonic numbers, and odd harmonic numbers. As special cases, we obtain evaluations of finite sums together with their corresponding infinite-series analogues. 
The approach provides a systematic mechanism for transforming identities for generalized binomial coefficients into identities for Catalan-number series.
\end{abstract}



\bigskip

\newpage

\section{Introduction}

Catalan numbers
\begin{equation}
C_n=\frac{1}{n+1}\binom{2n}{n},\qquad  n\ge0,\label{Catalan-Def}
\end{equation}
are among the most fundamental and extensively studied sequences in enumerative combinatorics. Their generating function is 
\begin{equation}\label{C_gen_fkt}
f(z) = \sum_{n=0}^\infty C_n z^n = \frac{1-\sqrt{1-4z}}{2z}, \qquad |z|<\frac14.
\end{equation}
They are closely related to the central binomial coefficients through numerous combinatorial and analytic identities.
Catalan numbers form the sequence A000108 in the OEIS \cite{OEIS}. They arise naturally in a wide variety of contexts, including lattice paths, binary trees, polygon triangulations, continued fractions, generating functions, and hypergeometric series. Because of their rich combinatorial and analytic structure, numerous identities, generating functions and summation formulas involving Catalan numbers have 
been established in the literature. For general background on Catalan numbers and their applications, we refer the reader to the monographs \cite{Koshy,Stanley} and the papers \cite{Alzer,Batir,Chu1,Chu2,Chu3,Deutsch,GSh-Indian,Mikic,Nagy,Witula}.

In recent years, considerable attention also has been devoted to finite and infinite series involving Catalan numbers and harmonic-type numbers. Such series occur naturally in combinatorial analysis, special functions and the evaluation of definite integrals. Closed-form evaluations are often obtained by means of generating functions, hypergeometric transformations, or binomial identities; see, for example, 
\cite{AFG-Axioms,AFG-Integers,AFG-Afr,AFG-Mex}. By a harmonic Catalan sum we  mean an combinatorial expression involving a product of Catalan numbers $C_n$ and harmonic numbers $H_n$ (or harmonic-type numbers). Here are two examples. First, we have \cite[Corollary 1]{AFG-Afr}
\begin{equation*}
\sum_{k=0}^n 4^{n-k} C_k H_{n-k}  = 2^{2n+1} H_{n+1} - (n+2)C_{n+1} O_{n+1},
\end{equation*}
where $O_n$ denotes the $n$th odd harmonic number, and also have \cite[Corollary 4.2]{Bataille}
\begin{equation*}
\sum_{k=1}^n \frac{2^{-2k}}{2k-1} \,C_k H_{k+1} = \frac{5}{9} - \frac{2^{-2n}}{3} \Big (H_{n+1} + \frac{2}{3}\Big )C_n.
\end{equation*}

The primary objective of this paper is to develop new methods for deriving finite and infinite identities of this type. 
We have two main lines of approach. First we develop a simple method for deriving identities involving Catalan numbers from generalized binomial coefficients with half-integer parameters. The key observation is that the coefficients $\binom{m+\frac12}{k+1}$ can be expressed in terms of Catalan numbers and products of odd linear factors. This representation enables a wide class of alternating binomial identities 
to be transformed into identities involving Catalan numbers. The second component starts from identities depending on complex parameters and applies differentiation with respect to these parameters to generate harmonic-number identities.

Using the first approach, we derive several families of finite and infinite series involving Catalan numbers, harmonic numbers, 
and odd harmonic numbers. In particular, we obtain explicit evaluations of sums containing factors of the form
$\prod\limits_{j=1}^{m}\big({2(k-j)+1}\big)$, together with related convolution-type identities involving Catalan and harmonic numbers. 
As a simple example, we prove that
\begin{equation*}
\sum_{k=0}^{\infty} \frac{2^{-2k}}{2k-1}\, C_k = -\frac23.
\end{equation*}
More generally, we evaluate series of the form
$$
\sum_{k=0}^{\infty} \frac{C_k}{2^{2k}} \prod_{j=1}^{m}\frac{1}{2(k-j)+1},
$$
and derive corresponding identities involving harmonic and odd harmonic numbers.

Binomial coefficients are defined, for nonnegative integers $i$ and $j$, by
\begin{equation*}
\binom ij=
\begin{cases}
\dfrac{{i!}}{{j!(i - j)!}}, & \text{$i \ge j$};\\
0, & \text{$i<j$}.
\end{cases}
\end{equation*}
The definition is extended to complex numbers $r$ and $s$ by
\begin{equation}\label{y89d722}
\binom rs= \frac{{\Gamma (r + 1)}}{{\Gamma (s + 1)\Gamma (r - s + 1)}},
\end{equation}
where the gamma function $\Gamma(z)$ is defined for $\Re(z)>0$ by
$
\Gamma (z) = \int_0^\infty  {e^{ - t} t^{z - 1}dt},
$
and is extended to the rest of the complex plane, excluding the nonpositive integers, by analytic continuation. 

Relationships between binomial coefficients can be established by using~\eqref{y89d722}, together with the following identities,  
which hold whenever the gamma functions involved are defined:
\begin{equation}\label{Gamma12+}
\Gamma \Big( {z + \frac{1}{2}} \Big) = \frac{\sqrt \pi}{2^{2z}}\binom{{2z}}{z}\,\Gamma ({z + 1}),
\end{equation}
and
\begin{equation}\label{Gamma12-}
\Gamma \Big( { - z + \frac{1}{2}} \Big) = \frac{(- 1)^z2^{2z}} {\binom{{2z}}{z}} \frac{{\sqrt \pi}}{{\Gamma (z + 1)}}.
\end{equation}

Throughout the paper, we use the harmonic numbers and the odd harmonic numbers
\begin{equation*}
H_n = \sum_{k=1}^{n} \frac{1}{k}\quad\,\text{and}\quad O_n = \sum_{k=1}^{n}\frac{1}{2k-1},
\end{equation*}
respectively, with $H_0=O_0=0$. The identity
\begin{equation}
H_{n+\frac12} = 2O_{n+1} - 2 \ln 2. \label{H-O-link}
\end{equation}
provides an explicit expression for harmonic numbers at half-integer arguments and follows from the relationship between the harmonic numbers, the odd harmonic numbers, and the digamma function.

Many of the finite and infinite identities established in this paper appear to be new and substantially extend the existing literature on Catalan-number identities.

The remainder of the paper is organized as follows. In Section~2, we derive identities from Pascal's formula and obtain several summation formulas involving Catalan numbers and odd harmonic numbers. In Section~3, we combine our approach with identities of Bat{\i}r and Sofo 
\cite{batirsofo23} to derive additional finite and infinite series involving Catalan numbers. Section~4 presents further applications of the generating-function approach, including new convolution identities for Catalan numbers, harmonic-number analogues of these identities and additional finite summation formulas.
In Section~5, we develop further consequences of the methods introduced above, focusing on identities involving generalized binomial coefficients
	and Catalan numbers. By specializing the parameters and differentiating the
	resulting identities, we derive several finite convolution formulas
	involving Catalan numbers, harmonic numbers, and odd harmonic numbers.

\section{Sums from Pascal's formula}

In this section we show how alternating binomial identities derived from Pascal's formula lead naturally to finite and infinite series involving Catalan numbers. The main ingredient is an explicit representation of generalized binomial coefficients with half-integer parameters in terms of Catalan numbers. The following identity provides the bridge between generalized binomial coefficients and Catalan numbers.
\begin{lemma}\label{rkm6e32}
If $r$ and $k$ are nonnegative integers, then
\begin{equation}\label{Lemma1}
\binom{r+\frac12}{k+1} = \frac{(-1)^{k+r}(2r+1)(r+1)!} {2^{2k+r+1}\prod\limits_{j=1}^{r}\bigl(2(k-j)+1\bigr)}\,C_kC_r.
\end{equation}

In particular, for $k\ge r\ge0$,
\begin{equation}\label{Lemma1_2}
\binom{r+\frac12}{k+1} = (-1)^{k+r} \frac{2r+1}{2^{2k+1}} \frac{\binom{k}{r}}{\binom{2k}{2r}}\,C_k. 
\end{equation}
\end{lemma}
\begin{proof}
From the definition of the generalized binomial coefficient,
\begin{equation*} \binom{r+\frac12}{k+1} = \frac1{(k+1)!} \prod_{j=0}^{k}\Big(r-j+\frac12\Big),
\end{equation*}
we obtain
\[
 \binom{r+\frac12}{k+1} = \frac{(2r+1)!!}{2^{k+1}(k+1)!\big(2(r-k)-1\big)!!}.
\]
Using double-factorial formulas
$(2n+1)!!=\frac{(2n+1)!}{2^n n!}$ and $(2n-1)!!=\frac{(2n)!}{2^n n!}$, together with \eqref{Catalan-Def}
we readily obtain \eqref{Lemma1}. To derive \eqref{Lemma1_2}, we use
\begin{equation}\label{prod-double}
	\prod_{j=1}^{r}\bigl(2(k-j)+1\bigr) = \frac{(2k-1)!!}{(2k-2r-1)!!} = \frac{(2k)!}{2^r(2k-2r)!}\,\frac{(k-r)!}{k!},
\end{equation}
and \eqref{Catalan-Def}. Substituting these identities into \eqref{Lemma1} and simplifying completes the proof.
\end{proof}

The following lemma evaluates the product of consecutive odd integers appearing in Lemma \ref{rkm6e32}.
\begin{lemma}\label{vfidc73}
If $m$ and $n$ are nonnegative integers with $m \le n$, then
\begin{equation*}
\prod_{j = 1}^m \!\big(2(n - j) + 1\big) = \frac{m!}{2^m}\binom{n+1}{m}\frac{C_n}{C_{n - m}}.
\end{equation*}
\end{lemma}
\begin{proof}
Since \eqref{prod-double} holds, the result follows immediately from the formula $(2n-1)!!=\frac{(2n)!}{2^n n!}$ in conjunction  
with \eqref{Catalan-Def}. 
\end{proof}

We now combine the previous two lemmas with a simple alternating binomial identity obtained from Pascal's formula.
\begin{proposition}
If $m$ and $n$ are nonnegative integers with $m\leq n$, then
\begin{equation}\label{fjun0w9}
\sum_{k = 0}^n \frac{2^{-2k}C_k}{\prod\limits_{j = 1}^m \!{\big(2(k - j) + 1\big)}} = \frac{2^{m + 2}}{(m + 2)!}\frac{(- 1)^m}{C_{m + 1}} 
- \frac{( 2(n -m) + 1) 2^{m-2n}}{m!\,( 2m + 1 )\binom{n+1}{m}}\, C_{n - m}.
\end{equation}
\end{proposition}
\begin{proof}
Summing Pascal's identity
\begin{equation}\label{pascal}
\binom r{k+s}=\binom{r-1}{k+s}+\binom{r-1}{k+s-1},
\end{equation}
with alternating signs yields the telescoping relation
\begin{equation}\label{cmaqeg5}
\sum_{k = 0}^n ( - 1)^k \binom{{r}}{{k + s}} = ( - 1)^n \binom{r-1}{n + s} + \binom{r-1}{s-1}.
\end{equation}
Setting $s=1$ gives
\begin{equation}\label{vtvra8b}
\sum_{k = 0}^n ( - 1)^k \binom{{r}}{{k + 1}} = ( - 1)^n \binom{r-1}{n + 1} + 1.
\end{equation}

Finally, substituting $r=m+\frac12 $ and applying Lemmas \ref{rkm6e32} and \ref{vfidc73} produces \eqref{fjun0w9}.
\end{proof}
\begin{corollary}
The following identities hold:
\begin{gather}
\sum_{k = 0}^n \frac{C_k}{2^{2k}} = 2-\frac{2n+1}{2^{2n}}\,C_n, \label{ysya1y8} \\
\sum_{k = 1}^n (-1)^{n-k}\binom{n+1}{k} \frac{2^{2k}}{(k+1)\,C_k} = \frac{2}{2n+1} \Big(\frac{2^{2n}}{C_n}-(-1)^n(n+1)\Big),\label{m=n} \\
\sum_{k = 0}^\infty \frac{C_k}{2^{2k}\prod\limits_{j = 1}^m \!\big(2(k-j)+1\big)}	= \frac{(-1)^m2^{m+2}}{(m+2)!\,C_{m+1}},\label{infty}
\end{gather}
and
\begin{equation}\label{ulj6hfu}
	\sum_{k = 0}^\infty \frac{C_k}{2^{2k}} = 2. 
\end{equation}
\end{corollary}
\begin{proof}
Formula \eqref{ysya1y8} is obtained from \eqref{fjun0w9} by setting $m=0$. Formula \eqref{m=n} follows from \eqref{fjun0w9} by using \eqref{prod-double} together with the evaluations \eqref{Gamma12+} and \eqref{Gamma12-}. Formula \eqref{infty} is obtained from \eqref{fjun0w9} by letting $n\to\infty$, while formula \eqref{ulj6hfu} follows from \eqref{infty} by setting $m=0$.
\end{proof}
\begin{remark} 
Identities \eqref{ysya1y8} and \eqref{ulj6hfu} are well known; however, they are often expressed in terms of central binomial coefficients or in slightly different forms; see, for example, \cite[p.~11]{Cao}, \cite[p.~311]{Lang}, \cite[Theorem 22]{Qi-Guo}. The former is a finite analogue of the latter. Moreover, the infinite identity follows directly from the generating function of the Catalan numbers \eqref{C_gen_fkt} by evaluating it at $x=\frac14$.
\end{remark}
\begin{proposition}
If $m$ and $n$ are nonnegative integers with $m\leq n$, then
\begin{align}
&\sum_{k = m}^n \frac{2^{-2k}C_k O_{k - m} }{\prod\limits_{j = 1}^m \big(2(k - j) + 1\big)}  + \sum_{k = 0}^m \frac{2^{-2k}C_k O_{m - k} }{\prod\limits_{j = 1}^m \big(2(k - j) + 1\big)} \notag\\
&\qquad\quad  =  \frac{{(-1)^m 2^{m+2}} }{( {m + 2} )!}\frac{O_{m + 1} }{C_{m + 1} } - \frac{2^{m-2n} }{m!(2m+1)}\frac{C_{n-m}}{\binom{n+1}{m}}\Big(\big(2(n-m)+1\big)O_{n-m}+\frac{2(n+1)}{2m + 1}\Big) .\label{zn6vx2y}
\end{align}
\end{proposition}
\begin{proof}
Differentiating \eqref{cmaqeg5} with respect to $r$ and using
\begin{equation}\label{Dif_up}
\frac{d}{d m}\binom{m}{n} = \binom{m}{n}(H_m-H_{m-n}),
\end{equation}
we obtain, after some algebraic manipulations,
\begin{align*}
\sum_{k = 0}^n (- 1)^k \binom{r}{k + s}H_{r - k - s} = (- 1)^n \binom{r-1}{n+s}\Big(H_{r-n-s-1}+\frac{1}{r} \Big) 
+ \binom{r-1}{s-1}\Big(H_{r - s} + \frac{1}{r} \Big).
\end{align*}

Evaluation at $r=m+\frac12$ and $s=1$ yields
\begin{align*}
&\sum_{k = 0}^n {( - 1)^k \binom{m + \frac12}{k + 1}H_{m - k - \frac12} }\\ 
&\qquad\qquad\qquad= ( - 1)^{n} \binom{m - \frac12}{n + 1} \Big( H_{m-n-\frac32} +\frac{2}{2m+1}\Big) + H_{m - \frac12}+\frac{2}{2m+1},
\end{align*}
which produces \eqref{zn6vx2y} after applying Lemmas \ref{rkm6e32} and \ref{vfidc73} and using the evaluation \eqref{H-O-link}.
\end{proof}
\begin{corollary} 
The following identities hold:
\begin{gather}
\sum_{k = 1}^n \frac{{C_k O_k }}{{2^{2k} }} = 2 - \frac{C_n}{2^{2n}} \big( (2n + 1) O_n +2n+2 \big), \label{Cor2_1} \\
\sum_{k = 1}^n (-1)^{k-1} 2^{2k} \binom{n+1}{k+1}\frac{O_k}{C_k} = \frac{2n(n+1)}{(2n-1)^{2}}, \label{Cor2_1-2} 
\\
\sum_{k = 1}^n 
\left(\frac{2^{-2k}C_{n+k}}{\prod_{j = 1}^n \big({2(n-j+k) + 1}\big)}  + \frac{2^{2k}C_{n-k}}{\prod_{j = 1}^n \big(2(n-j-k) + 1\big)} \right)\!{O_k}\notag \\
 = \frac{(-1)^n 2^{3n+2}}{(n+2)!}\frac{O_{n+1}}{C_{n+1}} - \frac{2^{-n}\,O_{n}}{n!(2n+1)} 
-\frac{2^{1-n}}{n!(2n+1)^2},\label{Cor2_2}\\
\sum_{k = 0}^\infty {2^{-2k}\frac{{C_{k} O_{k} }}{\binom{k+m}{m-1}}} 
= \frac{1}{m(m+1)} \sum_{k = 1}^{m} (-1)^{k-1} 2^{2k}\binom{m+1}{k+1}\frac{O_{k}}{C_{k}},\label{Cor2_3}
\end{gather}
and
\begin{equation}\label{Cor2_4}
\sum_{k = 0}^\infty \frac{C_k O_k}{2^{2k}} = 2.
\end{equation}
\end{corollary}
\begin{proof}
Equations \eqref{Cor2_1} and \eqref{Cor2_1-2} follow by specializing \eqref{zn6vx2y} to $m=0$ and $m=n$, respectively. 
	Likewise, \eqref{Cor2_2} is obtained by setting $m=2n$ in 
\eqref{zn6vx2y}. Letting $n$ tend to infinity in \eqref{zn6vx2y} and using Lemma \ref{vfidc73} with $n$ replaced by $k+m$, we obtain 
\begin{equation*}
\prod_{j=1}^{m}\big(2(k+m-j)+1\big) = \frac{m!}{2^m} \binom{k+m+1}{m}\frac{C_{m+k}}{C_k}.
\end{equation*}
This yields \eqref{Cor2_3}, from which \eqref{Cor2_4} follows by setting $m=0$.
\end{proof}

We now derive another family of identities involving central binomial coefficients and Catalan numbers. The derivation relies on an explicit evaluation of a generalized binomial coefficient together with a telescoping identity for binomial coefficients. These auxiliary results lead to closed-form evaluations of several finite and infinite series.
\begin{lemma}\label{zd6f3yu}
	If $k$ is a complex number and $r$ is nonnegative integer, then
	\begin{equation*}
	\binom{{k - \frac12}}{{k + r}} = \frac{( - 1)^r}{ 2^{2r+2k}} \frac{\binom{{2r}}{r} \binom{{2k}}{k}}{\binom{{k + r}}{r}}.
	\end{equation*}
\end{lemma}
\begin{proof}
By the definition of the generalized binomial coefficient,
\[ 
\binom{k-\frac12}{k+r} = \frac{\Gamma\!\left(k+\frac12\right)}{\Gamma(k+r+1)\Gamma\!\left(\frac12-r\right)}.
\]
Using \eqref{Gamma12+} with $z=k$ and $z=r$ together with Euler's reflection formula
$$\Gamma\Big(\frac12+r\Big) \Gamma\Big(\frac12-r\Big) = (-1)^r\pi,$$
we obtain
$$\binom{k-\frac12}{k+r} = (-1)^r \frac{\Gamma(2r+1)\Gamma(2k+1)}{2^{2r+2k}\Gamma(r+1)\Gamma(k+1)\Gamma(k+r+1)}.$$
Finally, using
$$\binom{2k}{k}=\frac{\Gamma(2k+1)}{\Gamma(k+1)^2}\quad \text{and}\quad
\binom{k+r}{r}=\frac{\Gamma(k+r+1)}{\Gamma(r+1)\Gamma(k+1)},$$ the desired identity follows after straightforward simplification.
\end{proof}

We now derive another family of identities involving central binomial coefficients. The starting point is an explicit evaluation of the generalized binomial coefficient $\binom{k-\frac12}{k+r}$, which, together with a simple summation identity, leads to closed-form evaluations of finite and infinite series.

We also require the following summation identity, which is an immediate consequence of Pascal's formula.
\begin{lemma}
If $r$ and $s$ are complex numbers, then
\begin{equation}\label{j54gqp1}
\sum_{k = 0}^n \binom{{k + r}}{s} = \binom{n+r+1}{s+1} - \binom{{r}}{{s + 1}}.
\end{equation}
\end{lemma}
\begin{proposition}
If $n$ is a nonnegative integer and $m$ is a complex number, then
\begin{equation}\label{od0z56l}
\sum_{k = 0}^n 2^{-2k} \frac{\binom{2k}{k}}{\binom{k+m}{m}} = \frac{1}{2m - 1}\left (2m - \frac{2n+1}{2^{2n}} \frac{\binom{2n}{n}}{\binom{n+m}{m}} \right)\!.
\end{equation}
\end{proposition}
\begin{proof}
Set $r=-\frac12$ and $s=-\frac12-m$ in \eqref{j54gqp1} and use Lemma \ref{zd6f3yu}.
\end{proof}
\begin{corollary}\label{Cor3} 
The following identities hold:
\begin{gather*}
\sum_{k = 0}^n 2^{- 2k} \binom{{2k}}{k} = \frac{2n + 1}{2^{2n}}\binom{2n}{n}, \\
\sum_{k = 0}^n \frac{C_k}{2^{2k}} = 2 - \frac{2n + 1}{2^{2n}}\,C_n, 
\end{gather*}
and
\begin{equation*}
\sum_{k = 0}^\infty 2^{- 2k} \frac{\binom{2k}{k}}{\binom{k + m}{m}} = \frac{2m}{2m - 1}.
\end{equation*}
\end{corollary}
\begin{proof}
The first identity follows from \eqref{od0z56l} by setting $m=0$. 

The second is obtained from the first by using \eqref{Catalan-Def}.
Finally, letting $n\to\infty$ gives the last identity, since  $2^{-2n}{\binom{2n}{n}}{\binom{n+m}{m}}^{-1}\to 0.$
\end{proof}
\begin{proposition}\label{Prop4}
	If $m$ is a complex number and $n$ is a nonnegative integer, then
	\begin{align*}
	&\sum_{k = 0}^n {2^{2(n - k)} \frac{\binom{{2k}}{k}}{\binom{{k + m}}{m}} ( H_m - {H_{k + m} } )}\nonumber\\ 
	&\qquad\qquad =   \frac{2(2n + 1)}{ {(2m - 1)}^2 } \frac{\binom{2n}{n}}{\binom{{n + m}}{m}} + \frac{ {2n + 1}}{2m - 1} \frac{\binom{2n}{n}}{\binom{{n + m}}{m}}( {H_{n + m}  - H_m } ) - \frac{2^{2n+1}}{{\left( {2m - 1} \right)^2 }}.
	\end{align*}
\end{proposition}
\begin{proof}
Differentiate \eqref{od0z56l} with respect to $m$, using \eqref{Dif_up}.
\end{proof}
\begin{corollary}\label{Cor4} 
The following identities hold:
\begin{align}
\sum_{k = 0}^n 2^{- 2k} \binom{{2k}}{k} H_k &= \frac{{2n + 1}}{{2^{2n} }}\binom{{2n}}{n}( H_n  - 2) + 2,\label{Cor4-1}\\
\sum_{k = 0}^n 2^{- 2k} C_k H_{k+1} &= 4 - \frac{2n + 1}{2^{2n}}\,C_n ( H_{n + 1} + 2),\label{Cor4-2}
\end{align}
and
\begin{equation}
\sum_{k = 0}^n {O_k } = \frac{2n+1}{2}\,O_n - \frac{n}{2}.\label{Cor4-3}
\end{equation}
\end{corollary}
\begin{proof}
The first identity follows by setting $m=0$ in the proposition. Setting $m=1$ gives
\[ \sum_{k=0}^n\frac{C_k}{2^{2k}}\bigl(H_{k+1}-1\bigr) = 2-\frac{2n+1}{2^{2n}}\,C_n(H_{n+1}+1).\]
Adding
\[ \sum_{k=0}^n\frac{C_k}{2^{2k}} = 2-\frac{2n+1}{2^{2n}}\,C_n,\]
yields \eqref{Cor4-2}. 
The last identity is obtained by setting $m=-\frac12$ in \eqref{Prop4} and then applying the identities 
\eqref{H-O-link}  
and $\binom{n-1/2}{-1/2} = 2^{-2n}\binom{2n}{n}$.
\end{proof}
\begin{remark}
Identity \eqref{Cor4-1} in Corollary \ref{Cor4} is not new. It first appeared recently in \cite[Theorem 2]{Adegoke}, where it was generalized in a different way.
\end{remark}
\begin{remark}
The ``inverse'' harmonic sums
$
\sum\limits_{k=0}^n 2^{2k} {\binom{m+k}{k}}{\binom{2k}{k}}^{-1} H_{k+m}
$
have been studied by Bataille and Frontczak in {\rm\cite{Bataille}}. In addition, similar sums alternating in signs can be found 
in the articles by Chu \cite{Chu} and Jin and Du \cite{Jin}. For instance, the following identity holds:
\begin{equation*}
\sum_{k=0}^n (-1)^k \frac{\binom{n}{k}}{\binom{m+k}{k}} H_k = \frac{m}{n+m} ( H_{m-1} - H_{n+m-1}).
\end{equation*}
\end{remark}

Our next goal is to obtain finite summation formulas involving central binomial coefficients. To this end, we first express the generalized binomial coefficient $\binom{k}{r+\frac12}$ in terms of central binomial coefficients.
\begin{lemma}\label{x6lkv24}
If $r$ and $k$ are nonnegative integers, then
\begin{equation*}
\binom{{k}}{r + \frac12} = \frac{2^{r + 2}}{(r + 2)!\,\pi\, C_{r + 1}}\frac{2^{2k}}{\binom{{2k}}{k}}\prod_{j = 1}^r \big(2(k - j) + 1\big).
\end{equation*}
\end{lemma}
\begin{proof}
	By the Gamma-function representation,
	$$
	\binom{k}{r+\frac12}
	=\frac{\Gamma(k+1)}
	{\Gamma\big(r+\frac32\big)
		\Gamma\big(k-r+\frac12\big)}.
	$$
	Using
\[
	\Gamma\Big(k+\frac12\Big)
=\Gamma\Big(k-r+\frac12\Big)
	\prod_{j=1}^r\Big(k-j+\frac12\Big),
\]
	Legendre's duplication formula, and
	$
	\Gamma\big(r+\frac32\big)
	=
	\frac{(2r+2)!}{4^{r+1}(r+1)!}\sqrt{\pi},
	$
	we obtain
	\[
	\binom{k}{r+\frac12}
	=
	\frac{2^{r+2}(r+1)!}{\pi(2r+2)!}
	\frac{2^{2k}}{\binom{2k}{k}}
	\prod_{j=1}^r\bigl(2(k-j)+1\bigr).
	\]
	Since
$\frac{(r+1)!}{(2r+2)!}
	=	\frac{1}{(r+2)!C_{r+1}}$,
	the result follows.
\end{proof}
\begin{proposition}
If $m$ and $n$ are nonnegative integers with $m\leq n$, then
\begin{align}\label{qjck4vx}
\sum_{k = 0}^n \frac{2^{2k}}{\binom{2k}{k}} \prod_{j = 1}^m \big(2(k - j) + 1 \big) 
= \frac{2^{2n-m+1}m!}{2m+3}\,\frac{\binom{n + 1}{m}}{C_{n - m}} +  \frac{(- 1)^m (m + 2)!}{(2m + 3)2^{m + 1}}\,C_{m + 1}.
\end{align}
\end{proposition}
\begin{proof}
Setting $r=0$ and $s=m+\frac12$ in \eqref{j54gqp1} gives
\begin{equation*}
\sum_{k = 0}^n \binom{{k}}{{m + \frac12}} = \frac{2(n + 1)}{2m + 3}\binom{{n}}{{m + \frac12}} + \frac{( - 1)^m 2}{\pi (2m + 3)},
\end{equation*}
and hence \eqref{qjck4vx} upon using Lemma \ref{x6lkv24}.
\end{proof}
\begin{corollary}\label{cb_cor}
The following formulas hold:
\begin{gather}
\sum_{k = 0}^n \frac{2^{2k}}{\binom{{2k}}{k}} = \frac{{2^{2n + 1}}}{{3C_n }} + \frac{1}{3},\label{Cor5-1}\\
\sum_{k = 0}^n k\frac{ 2^{2k}}{\binom{{2k}}{k}} = \frac{{(3n+1)2^{2n + 1}}}{{15C_n }} - \frac{2}{15},\label{Cor5-2}
\end{gather}
and
\begin{gather}
\sum_{k = 0}^n k^2\frac{ 2^{2k}}{\binom{{2k}}{k}} = \frac{(15n^2+12n-1)2^{2n + 1}}{105C_n}+ \frac{2}{105}.\label{Cor5-3}
\end{gather}
\end{corollary}
\begin{proof}
Setting $m=0$ in \eqref{qjck4vx} gives the first identity. For $m=1$, using $C_{n-1}=\frac{n+1}{2(2n-1)}\,C_n$ together with \eqref{Cor5-1} yields \eqref{Cor5-2}. Finally, for $m=2$, $\prod_{j=1}^{2}\bigl(2(k-j)+1\bigr)=4k^2-8k+3.$ Therefore,
$$4\sum_{k=0}^n \frac{k^2 2^{2k}}{\binom{2k}{k}} - 8\sum_{k=0}^n \frac{k2^{2k}}{\binom{2k}{k}} + 3\sum_{k=0}^n\frac{2^{2k}}{\binom{2k}{k}}
= \frac{2^{2n-1}}{7}\frac{n(n+1)}{C_{n-2}} + \frac{12}{7}.$$
Since $C_{n-2}=\frac{n(n+1)}{4(2n-1)(2n-3)}\,C_n$, substituting the first two identities and simplifying gives \eqref{Cor5-3}.
\end{proof}
\begin{remark}
The three identities in Corollary \ref{cb_cor} are all known. The first sum is a classical and has appeared in many papers. 
The remaining two identities are stated in \cite[Corollary 3.2]{Bataille}. Moreover, this reference also treats the more general family of sums
$\sum\limits_{k = 0}^n {2^{2k}}{\binom{2k}{k}}^{-1} k^d$,
for every nonnegative integer $d$; see (\cite[Theorem 3.1]{Bataille}). 
\end{remark}

\section{Sums from two identities of Bat\i r and Sofo}

In this section, we derive several identities involving Catalan numbers from two recent binomial identities due 
to Bat\i r and Sofo \cite{batirsofo23}. By specializing the parameter in their formulas and applying the auxiliary results 
established in the previous section, we obtain explicit evaluations and convolution identities for Catalan numbers.
\begin{proposition}
If $m$ and $n$ are nonnegative integers with $m\leq n$, then
\begin{align}
&\sum_{k = 0}^{n - 1} \frac{2^{2(n-k)}}{ n - k}\,\frac{C_k}{\prod\limits_{j = 1}^m \big(2(k - j) + 1\big)} 
\notag\\
&\qquad\qquad\qquad= \frac{2^{m + 1}}{m!\binom{n+1}{m}}\, C_{n - m}( O_{n - m} - O_{m + 1}) + \frac{{( - 1)^m 2^{2n+m + 2} }}{( {m + 2} )!(n + 1)C_{m + 1} }\label{hv4b6v2}.
\end{align}
\end{proposition}
\begin{proof}
Bat\i r and Sofo derived an identity \cite[Theorem 2.2]{batirsofo23}
\begin{equation*}
\sum_{k = 0}^{n - 1} \frac{(-1)^{k-1}}{n-k} \binom{r}{k} = ( - 1)^n \binom{r}{n} \big( \psi(r+1) - \psi(r+1-n)\big),
\end{equation*}
where $\psi(x)=\frac{\Gamma'(x)}{\Gamma(x)}$ is the digamma function.  
After shifting the summation index and using
\begin{equation}
H_r = \psi(r+1)+\gamma, \label{Hn-psi}
\end{equation}
where $\gamma = \lim\limits_{n\to \infty} \Big(\sum_{k=1}^n\frac1{k}-\ln n\Big)$ is the Euler--Mascheroni constant, we obtain
\begin{equation}\label{t4xcru7}
\sum_{k = 0}^{n - 1} \frac{(- 1)^k}{n - k}\binom{r}{k + 1} = (- 1)^n \binom{r}{n + 1}(H_{r - n - 1} - H_r) + \frac{1}{n + 1}.
\end{equation}

Setting $r=m+\frac12$ in \eqref{t4xcru7} and using Lemmas \ref{rkm6e32} and \ref{vfidc73}, and Eq. \eqref{H-O-link} produces \eqref{hv4b6v2}.
\end{proof}
\begin{corollary} 
The following identities hold:
\begin{gather*}
\sum_{k = 1}^n \frac{2^{2k}}{k}\,C_{n - k} = \frac{{2^{2n + 1} }}{n + 1} + 2C_n ( O_n  - 1),\\
\sum_{k = 0}^{n-1} \frac{C_k }{2^{2k} (n-k)(2k-1)} = \frac{2^{2-2n}}{3(n+1)}\,C_{n-1}(3O_{n-1}-4)- \frac{2}{3(n+1)},
\end{gather*}
and
\begin{gather*}
\sum_{k = 1}^{n} (-1)^{k-1}\frac{2^{2k}}{k}\,\frac{\binom{n+1}{k+1}}{C_k} = 2(n+1)O_{n}.
\end{gather*}
\end{corollary}
\begin{proof}
The identities follow from \eqref{hv4b6v2} by setting $m=0$, $m=1$, and $m=n$, respectively. For the third identity, we also use the relation
\begin{equation}\label{prod(2k-2j+1)} \prod_{j=1}^{n}\big({2(k-j)+1}\big) = \frac{(-1)^{n-k}(n+2)!}{2^n\binom{n+2}{k+1}}\,C_kC_{n-k}, \qquad 0\le k\le n.
\end{equation}
This completes the proof.
\end{proof}
\begin{proposition}
If $m$ and $n$ are nonnegative integers, then
\begin{align}
&\sum_{k = 1}^{n} \frac{2^{2k-1}  C_{n-k}H_{k - 1}}{k\prod\limits_{j = 1}^m \!{\big(2(n- k - j) + 1\big)}}  \notag\\
&\qquad\qquad= \frac{(-1)^m 2^{2n+m}}{(2m+1)(m+1)!(n+1)}\frac{H_n}{C_m} + \frac{(O_{m-n}-O_{m+1})^2+O^{(2)}_{m-n}-O^{(2)}_{m+1}}{{\prod\limits_{j = 1}^m {\big( {2( {n - j} ) + 1} \big)} }}\,C_n.\label{Prop6_1}
\end{align}
\end{proposition}
\begin{proof} Substituting $r=m+\frac{1}{2}$ into the identity \cite[Theorem 2.4]{batirsofo23},
\begin{equation*}
		\sum_{k=1}^n (-1)^k \binom{r}{n-k}\frac{H_{k-1}}{k}=
		\frac{1}{2}\binom{r}{n}
		\!\left(
		\bigl(\psi(r+1)-\psi(r+1-n)\bigr)^2
		+\psi'(r+1)-\psi'(r+1-n)
		\right)\!,
\end{equation*}
		we obtain
		\begin{align*}
		&\sum_{k=1}^n (-1)^k
		\binom{m+\frac12}{n-k}\frac{H_{k-1}}{k}\\
		&\qquad=
		\frac{1}{2}\binom{m+\frac12}{n}\!
		\left(
		\left(
		\psi\Big(m+\frac32\Big)
		-\psi\Big(m+\frac32-n\Big)
		\right)^2
		+\psi'\Big(m+\frac32\Big)
		-\psi'\Big(m+\frac32-n\Big)
		\right).
		\end{align*}
		Now, using \eqref{Hn-psi} and the identities
		\begin{gather*}
		\psi(x)-\psi(x-n)=
		2\bigl(O_{x-\frac12}-O_{x-n}\bigr),\\
		\psi'(x)-\psi'(x-n)
		=
		4\bigl(O_{x-n-\frac12}^{(2)}-O_{x-\frac12}^{(2)}\bigr),
		\end{gather*}
		together with Lemma \ref{rkm6e32} twice for the half-integer binomial coefficients, we obtain, after straightforward simplification, \eqref{Prop6_1}.
\end{proof}
\begin{corollary} 
The following identities hold:
	\begin{gather*}
	\sum_{k = 1}^n \frac{2^{2k}}{k}\,C_{n - k} H_{k-1} = \frac{2^{2n + 1}}{n + 1} H_{n} + 2C_n\big( O^2_{n} - 2O_{n} + O^{(2)}_{n}\big),\\
	\sum_{k = 1}^{n} \frac{2^{2k}}{k\big(2(n-k)-1\big)}\,C_{n-k} H_{k-1} = - \frac{2^{2n+1}}{3(n+1)}H_n + \frac{4C_{n-1}}{n+1}
	\Big( O^2_{n-1} - \frac83 O_{n-1} -  O^{(2)}_{n-1} + \frac23 \Big), 
  \end{gather*}
 and 
  \begin{gather*}
  \sum_{k = 1}^{n} (-1)^k\binom{n}{k}\frac{2^{2k-1}}{k(k+1)}\frac{H_{k-1}}{C_k} = O^2_{n} - O^{(2)}_{n}.	
\end{gather*}
\end{corollary}
\begin{proof}
	The identities follow from \eqref{Prop6_1} by setting $m=0$, $m=1$, and $m=n$, respectively. We also use the identities $O_{-n}=O_n$ and  $O^{(2)}_{-n}=-O^{(2)}_n$; see \cite{KA_Cubic}. For the third identity, we additionally use \eqref{prod(2k-2j+1)} and $ \binom{n+\frac12}{n+1} = \frac{n+2}{2^{2n+2}}\,C_{n+1}$.
\end{proof}

\section{Sums associated with the binomial theorem}

We now turn our attention to finite binomial sums. The following auxiliary lemma provides an explicit expression for the generalized binomial coefficient $\binom{k}{r+\frac12}$, which, combined with Frisch's identity, enables us to derive several explicit summation formulas.

We require the following variation on Lemma \ref{x6lkv24}.
\begin{lemma}\label{Lem5}
If $r$ and $k$ are nonnegative integers, then
\begin{equation}\label{n2c8e0k}
\binom{k}{r+\frac12} = \frac{2^{2k+r+1}}{(r+2)!\,\pi\,C_{k-1}C_{r+1}} \prod_{j=1}^{r-1}\big(2(k-j)-1\big).
\end{equation}
\end{lemma}

Multiplying the binomial theorem
\begin{equation*}
\sum_{k = 0}^n {( - 1)^k \binom{{n}}{k}x^k }  = \left( {1 - x} \right)^n 
\end{equation*}
by $x^{r-s}(1-x)^{s}$ and integrating termwise from $0$ to $1$ using Euler's Beta function gives Frisch's identity (also known as Gould's identity 4.2 \cite{Gould}), namely,
\begin{equation}\label{frisch}
\sum_{k = 0}^n {(-1)^k\frac{\binom nk}{\binom{k + r}s} = \frac s{n+s}\frac{1}{\binom{n+r}{n+s}}}, \qquad r \geq s.
\end{equation}
\begin{proposition}
If $s$ and $n$ are nonnegative integers, then
\begin{gather}\label{Prop8}
\sum_{k = 0}^{n} \frac{(-1)^k}{2^{2k}} \binom nk \frac {C_k}{\prod\limits_{j=1}^{s-1}\big(2(k-j)+1\big)}
= \frac{(-1)^{s} 2^{s-2n}(1-2s)}{(n+1)(s+1)!}\binom{n+s+1}{n}\frac{C_{n+s}}{C^2_s}.
\end{gather}
\end{proposition}
\begin{proof}
Setting $r=1$ and $s=m+\frac12$ in \eqref{frisch} and using Lemma \ref{Lem5} together with 
\begin{equation*}
\prod_{j=1}^{n+s-1}\big(2(n-j)+1\big) = (-1)^{s-1} \frac{(n+1)!\,s!\,C_n\,C_{s-1}}{2^{\,n+s-1}}, \qquad s\ge1.
\end{equation*}
This completes the proof. 
\end{proof}
\begin{corollary}
The following identities hold:
\begin{gather*}
\sum_{k = 0}^{n}(-1)^k \binom nk \frac{k C_k}{2^{2k}} = -\frac{n C_{n}}{2^{2n}},\\
\sum_{k = 0}^{n} (-1)^k \binom nk \frac{C_k}{2^{2k}} = \frac{2n+1}{2^{2n}}\, {C_{n}},
\end{gather*}
and
\begin{gather*}
\sum_{k = 1}^{n} \binom {n}{k}\binom {n+1}{k} \frac{2^{2k}}{C_{k-1}} = (2n+1)\frac{C_{2n}}{C_{n-1}} + 2.
\end{gather*}
\end{corollary}
\begin{proof} 
The identities follow from \eqref{Prop8} by setting $m=0$, $m=1$, and $m=n$, respectively. The details are straightforward 
and are therefore omitted.
\end{proof}
\begin{proposition}  
	If $m$ and $n$ are nonnegative integers, then
\begin{align}\label{nyznmfa}
&\sum_{k = 1}^n {\frac{{( - 1)^k }}{{2^{2k-1} }}\binom{{n}}{k}\frac{{C_{k-1} }}{{\prod\limits_{j = 1}^{m - 2} {\big( {2( {k - j}) - 1} \big)} }}}\nonumber\\
&\qquad\qquad  = (-1)^{m-1}\frac{(2m-1)^2\, 2^{m-2n}}{(2n+2m-1)(m+1)!} \binom{n + m + 1}{n}\frac{C_{n+m}}{C^2_{m}} + \frac{(-1)^m2^{m-1}}{m!C_{m-1}}
\end{align}	
and
\begin{align}\label{g2a6ug4}
&\sum_{k = 1}^n \frac{{( - 1)^k }}{2^{2k-1} }\binom{n}{k}\frac{( H_k  - 2O_{k - m + 1} )C_{k-1} }{\prod\limits_{j = 1}^{m - 2} {\big( {2( {k - j}) - 1} \big)}}\nonumber\\
&\qquad  = \frac{(-1)^m (2m-1)^2\,2^{m - 2n}}{(2n + 2m - 1)(m+1)!} \binom{n + m + 1}{n}  ( 2O_{m - 1}- H_n)\frac{C_{n +m}} {C_m^2} - \frac{(-2)^m }{m! }\frac{O_{m-1}}{C_{m-1}}.
	\end{align}
\end{proposition}
\begin{proof}
	Differentiate \eqref{frisch} with respect to $r$, then set $r=1$ and $s=m+\frac12$ in the resulting equation, and finally
	use \eqref{n2c8e0k} to obtain 
	\begin{align*}
	&\sum_{k =0}^{n} \frac{(-1)^k}{2^{2k}}\binom{n}{k} \frac{k+1}{2k-1}\,\frac{H_{k}-H_{k-m + \frac12}}{\prod\limits_{j=1}^{m-2}\big(2(k-j)-1\big)}\,C_k \\
	&\qquad\qquad=\frac{2^{1-3n}(2m-1)n!}{2n+2m-1}\binom{n+m+1}{n} \frac{H_n-H_{-m+\frac12}}{\prod\limits_{j=1}^{n+m-2}\big(2(n-j)-1\big)}\frac{C_{n-1}C_{n+m}}{C_{m}}.
	\end{align*}
Now write $$H_k  - H_{k - m + \frac12}  = H_k  - 2O_{k - m + 1}  + 2\ln 2$$ and $$H_n  - H_{ - m + \frac12}  = H_n  - 2O_{m - 1}  + 2\ln 2$$ and equate respective rational and irrational parts from both sides of the resulting equation to obtain \eqref{nyznmfa} and \eqref{g2a6ug4}.
\end{proof}
\begin{proposition} 
If $m$ and $n$ are nonnegative integers, then
\begin{align*}\label{hs6d8ge}
	&\sum_{k=1}^{n} \frac{(-1)^{k-1}}{2^{2k-2}}\binom{n}{k} \frac{O_m-O_{k-m+1}}{\prod\limits_{j=1}^{m-2}\big(2(k-j)-1\big)}\,C_{k-1}\notag \\
	&\qquad\qquad =\frac{(-1)^m 2^{m-2n}\binom{n+m+1}{n} }{m!(2n+2m-1)}
	\left(\frac{2n}{2n+2m-1} + (2m-1)\big(O_{n+m}-O_{m-1}\big)\right)\frac{C_{n+m}}{C_{m}C_{m-1}}\\
	&\qquad\qquad\quad-\frac{(-1)^m 2^{m+1}}{(m+1)!C_m}.
		\end{align*}
\end{proposition}
\begin{proof}
	Differentiating \eqref{frisch} with respect to $s$, then setting $r=1$ and $s=m+\frac32$ in the resulting equation, and using \eqref{n2c8e0k} completes the proof. In obtaining the final form, we used \eqref{H-O-link}.
\end{proof}

\section{Additional results}

In this section, we present several additional identities that follow naturally from the methods developed above. We first establish a convolution identity involving generalized binomial coefficients and Catalan numbers. By specializing its parameters and differentiating with respect to the parameter, we obtain several finite convolution formulas involving Catalan numbers, harmonic numbers, and odd harmonic numbers. These results provide further applications of the connection between generalized binomial coefficients and Catalan numbers developed in the preceding sections.
\begin{lemma}
For all complex numbers $x$ and all nonnegative $n$, we have
\begin{equation}\label{C_bin_id}
\sum_{k=0}^n \frac{(-1)^k}{2^{2k}}\binom{x}{n-k}  C_k = \frac{2n+1}{2^{2n}} \frac{\binom{2x+1}{2n+2}}{\binom{x}{n+1}} \,C_n - 2 \binom{x}{n+1}.
\end{equation}
\end{lemma}
\begin{proof}
	Let $a_n=(-1)^n 2^{-2n}C_n$ and $b_n=\binom{x}{n}$. Let further 
	$$A(z)=\sum_{n=0}^\infty a_n z^n \quad\text{and}\quad B(z)=\sum_{n=0}^\infty b_n z^n$$
	be their respective ordinary generating functions. Then, using \eqref{C_gen_fkt}, we have
	$$A(z) = f\left(-\frac{z}{4}\right ) = \frac{2}{z} \big (\sqrt{1+z}-1 \big ),$$
	and also from the generalized binomial theorem
	$B(z) = (1+z)^x$. Thus
	\begin{align*}
	A(z)\cdot B(z) &= \sum_{n=0}^\infty \sum_{k=0}^n (-1)^k 2^{-2k} C_k \binom{x}{n-k} z^n = \frac{2}{z} \left ((1+z)^{x+\frac12} - (1+z)^x \right ) \\
	&= \frac{2}{z} \sum_{n=0}^\infty \left ( \binom{x+\frac12}{n} - \binom{x}{n} \right ) z^n = 2 \sum_{n=0}^\infty \left ( \binom{x+\frac12}{n+1} - \binom{x}{n+1} \right ) z^n.
	\end{align*}
	Extracting the coefficients in the Cauchy product, we obtain
	\begin{equation*}
	\sum_{k=0}^n (-1)^k  \binom{x}{n-k} 2^{-2k} C_k = 2\left (\binom{x+\frac12}{n+1} - \binom{x}{n+1}\right )\!.
	\end{equation*}
	The statement now follows because
	$$\binom{x+\frac12}{n+1} = \binom{2x+1}{2n+2} \binom{2n+2}{n+1} \binom{x}{n+1}^{-1} 2^{-2n-2}.$$
\end{proof}
\begin{remark}
	If $x=0$, then in view of 
	$\lim\limits_{x\rightarrow 0} {\binom{2x+1}{2n+2}}{\binom{x}{n+1}}^{-1} = \frac{(-1)^n}{2n+1},$
	both sides of \eqref{C_bin_id} become $(-1)^n 2^{-2n} C_n$.
\end{remark}
\begin{corollary}
If $n$ is a nonnegative integer, then
	\begin{gather}
	\sum_{k = 0}^n \frac{( - 1)^k}{2^{2k}} \binom{{n}}{k} C_k = \frac{2n + 1}{2^{2n}} C_n \label{rk60elg}, \\ 
	\sum_{k = 0}^n \frac{ C_k}{2^{2k}} =  2-\frac{2n + 1}{2^{2n}}C_n ,\notag\\
	\sum_{k = 0}^n (k+1) C_{n - k} C_k = (2n+1)C_n,\notag
	\end{gather}
and 
\begin{equation}
\sum_{k = 1}^n C_{k-1}C_{n-k} = C_n,\qquad n\geq1.\notag
\end{equation}
\end{corollary}
\begin{proof} These are respective evaluations of \eqref{C_bin_id} at $x=n$, $x=-1$, $x=-\frac12$, and $x=\frac12$. In deriving \eqref{rk60elg} we used
	\begin{equation}\label{vmrppnb}
	\lim_{x\to n}\binom{{2x + 1}}{{2n + 2}}\binom{{x}}{{n + 1}}^{ - 1} = 1.
	\end{equation}
\end{proof}
\begin{theorem}
	If $x$ is a complex number, then we have
	\begin{align}\label{CH_bin_id}
	&\sum_{k=0}^n \frac{(-1)^k}{2^{2k}}\binom{x}{n-k}  C_k H_{x+k-n}  \nonumber \\
	&\qquad= \frac{2n+1}{2^{2n}} \frac{\binom{2x+1}{2n+2}}{\binom{x}{n+1}} \big ( 2H_x - H_{x-n-1} +2 H_{2x-2n-1} - 2H_{2x+1}\big )C_n 
	- 2 \binom{x}{n+1} H_{x-n-1}.  
	\end{align}
\end{theorem}
\begin{proof}
	Work with \eqref{C_bin_id} and differentiate both sides with respect to $x$ while using
	\begin{gather*}
	\frac{d}{dx} \binom{x}{n-k} = \binom{x}{n-k} (H_x - H_{x+k-n}  ), \\
	\frac{d}{dx} \binom{x}{n+1} = \binom{x}{n+1} (H_x - H_{x-1-n}  ), \\
	\frac{d}{dx} \frac{\binom{2x+1}{2n+2}}{\binom{x}{n+1}} = \frac{\binom{2x+1}{2n+2}}{\binom{x}{n+1}},(H_{x-n-1} - 2H_{2x-2n-1} 
	- H_x +2H_{2x+1}).
	\end{gather*}
	When simplifying use \eqref{C_bin_id} again.
\end{proof}
\begin{corollary}
We have
\begin{gather}
\sum_{k=0}^n \frac{(-1)^k}{2^{2k}} \binom{n}{k}  C_k H_{k} = \frac{2}{n+1} - \frac{2n+1}{2^{2n-1}}\,C_n (H_{2n+1}-H_n) \label{nwg1l9g},\\
\sum_{k=0}^n \frac{(-1)^k}{2^{2k}} \binom{n+1}{k+1} C_k H_{k+1} = \frac{(2n+1)(2n+3)}{2^{2n-1}}\, C_n (1 + H_{n+1} - H_{2n+3}),\label{uujgkj3}\\
\sum_{k = 0}^n (-1)^{n-k} \frac{\binom{2n+1}{2k}}{\binom{n+1}{k+1}} C_k C_{n-k} = \frac{2^{2n + 1}}{n + 1} - \frac{2n+1}{n+1}C_n, \label{yj0t8qz}
\end{gather}
and
\begin{equation}\label{nau3hwm}
\sum_{k = 0}^n (-1)^{k} \frac{\binom{2n+1}{2k+1}}{\binom{n+1}{n-k+1}} C_k C_{n-k} O_{k+1} = \frac{2^{2n + 1}}{n + 1} (2O_{n+1}-H_{2n+2}). 
\end{equation}
\end{corollary}
\begin{proof}
	Observe that identity \eqref{CH_bin_id} can also be written as
	\begin{align}\label{a2w87sw}
	&\sum_{k=0}^n \frac{(-1)^k}{2^{2k}} \binom{x}{n-k}  C_k H_{x+k-n}  \nonumber \\
	&\qquad\qquad\qquad= \frac{2n+1}{2^{2n-1}} \frac{\binom{2x+1}{2n+2}}{\binom{x}{n+1}}  \,C_n ( H_x + O_{x-n} - H_{2x+1}) - 2 \binom{x}{n+1} H_{x-n-1}.  
	\end{align}
	Identity \eqref{nwg1l9g} is a special case of identity \eqref{a2w87sw} for $x=n$, using \eqref{vmrppnb} and calculating
	$$\lim_{x\rightarrow n} \binom{x}{n+1} H_{x-n-1} = - \frac{\Gamma(n+1)}{\Gamma(n+2)} = - \frac{1}{n+1}.$$
	Identity~\eqref{uujgkj3} follows from \eqref{CH_bin_id} for $x=n+1$.
	
	To prove \eqref{yj0t8qz} and \eqref{nau3hwm}, set $x=n+\frac12$ in \eqref{a2w87sw}, use \eqref{H-O-link}, noting, in particular, 
	that $O_{\frac12}=\ln 2$; equate respective rational and irrational parts of the resulting equation to obtain
	\begin{equation*}
	\sum_{k = 0}^n \frac{(-1)^k}{2^{2k}} \binom{{n + \frac12}}{n - k}  C_k = \frac{2n + 1}{2^{2n}} \binom{{n + \frac12}}{{n + 1}}^{- 1} C_n - 2\binom{{n + \frac12}}{n + 1}
	\end{equation*}
and
	\begin{equation*}
	\sum_{k = 0}^n \frac{(- 1)^k}{2^{2k}} \binom{{n + \frac12}}{{n - k}}  C_k O_{k + 1} = 
	\frac{2n + 1}{2^{2n}} \binom{{n + \frac12}}{{n + 1}}^{- 1} C_n\, ( 2O_{n + 1} - H_{2n + 2}),
	\end{equation*}
from which \eqref{yj0t8qz} and \eqref{nau3hwm} follow since
	$	\binom{{n + \frac12}}{{n + 1}} = \frac{{2n + 1}}{{2^{2n + 1} }}C_n $
	and
	\begin{equation*}
\binom{n + \frac12}{k+\frac12} = \frac{1}{2^{2( n - k)} }
\binom{2n+1}{2(n-k)} \binom{{2( n - k)}}{n - k}\binom{{n}}{k}^{ - 1} .
\end{equation*}
\end{proof}
\begin{lemma}\label{lemma_Car}
	For all real (complex) numbers $x$ and nonnegative integer $n$, we have
	\begin{equation}\label{C_bin_id2}
	\sum_{k=0}^n \binom{2k+x}{k} C_{n-k} = \binom{2n+x+1}{n}.
	\end{equation}
\end{lemma}
\begin{proof}
	To prove the identity we apply the following result of Carlitz \cite[Equation (2.2)]{Carlitz}:
	$$\sum_{r=0}^\infty \binom{2r+x}{r}\frac{u^r}{(1+u)^{2r}} = \frac{(1+u)^{x+1}}{1-u}.$$

	Setting
	$\frac{u}{(1+u)^2} = y$, we obtain  $ u^2 + u \frac{2y-1}{y} + 1=0.$
	Solving this  equation for $u$ gives
	$$u_{1,2} = \frac{1 \pm \sqrt{1-4y}}{2y} - 1.$$

	The connection to Catalan numbers via \eqref{C_gen_fkt} becomes visible. We set $u_2 = f(y)-1$ and calculate
	$$\sum_{n=0}^\infty \binom{2n+x}{n} y^n - \sum_{n=1}^\infty C_n y^n\cdot \sum_{n=0}^\infty \binom{2n+x}{n} y^n 
	= \left (\sum_{n=0}^\infty C_n y^n \right )^{x+1}$$
	or
	$$
	2 \sum_{n=0}^\infty \binom{2n+x}{n} y^n - \sum_{n=0}^\infty C_n y^n \cdot\sum_{n=0}^\infty \binom{2n+x}{n} y^n 
	= \left (\sum_{n=0}^\infty C_n y^n \right )^{x+1}.$$
	Finally, we obtain
	$$2 \sum_{n=0}^\infty \binom{2n+x}{n} y^n - \sum_{n=0}^\infty \left ( \sum_{k=0}^n C_k \binom{2(n-k)+x}{n-k} \right ) y^n
	= \sum_{n=0}^\infty \frac{x+1}{2n+x+1} \binom{2n+x+1}{n} y^n,$$
	where on the right-hand side we have applied a result of Deutsch and Shapiro \cite{Deutsch}: if $[y^n] f=C_n,$ where $f$ is given 
	in \eqref{C_gen_fkt}, then for any positive integer $q$, $$[y^n]f^q=\frac{q}{2n+q}\binom{2n+q}{n}.$$ 
	Extracting the coefficients of $y^n$
	completes the proof.
\end{proof}
\begin{theorem}
If $x$ is a complex number, then we have
\begin{align}
	&\sum_{k=0}^n \binom{2k+x}{k} (H_{2k+x} - H_{k+x}) C_{n-k} \notag\\
&\qquad\qquad\qquad\qquad =  \binom{2n+x+1}{n} (H_{2n+x} - H_{n+x}) - \frac{n}{(n+x+1)^2} \binom{2n+x}{n}.\label{CH_bin_id2}
\end{align}
\end{theorem}
\begin{proof}
	Work with \eqref{C_bin_id2} and differentiate both sides with respect to $x$ while using
	\begin{equation*}
	\frac{d}{dx} \binom{2n+x}{n} = \binom{2n+x}{n} (H_{2n+x} - H_{n+x} ).
	\end{equation*}
\end{proof}
\begin{corollary}
	We have
	\begin{gather}\label{yfkqk6b}
	\sum_{k=1}^n \binom{2k}{k} (H_{2k} - H_{k}) \,C_{n-k} = (2n+1) C_n (H_{2n} - H_{n}) - \frac{n}{n+1}\, C_n,\\
	\label{bxuu064}
	\sum_{k=1}^n \binom{2k}{k} (H_{2k-1} - H_{k-1})\, C_{n-k} = 2 \,\binom{2n}{n} \Big ( H_{2n-1} - H_{n-1} - \frac{1}{2n} \Big ),
	\end{gather}
	and
	\begin{equation}\label{pzfs47k}
	\sum_{k = 1}^n \binom{2k}{k} \frac{C_{n - k}}{k} = 2C_n (H_{2n} - H_{n + 1} + 1 ).
	\end{equation}
\end{corollary}
\begin{proof}
Set $x=0$ and $x=-1$ in \eqref{CH_bin_id2} to obtain, in turn, \eqref{yfkqk6b} and \eqref{bxuu064}. 
Identity \eqref{pzfs47k} follows after subtracting \eqref{yfkqk6b} from \eqref{bxuu064} and some algebra.
\end{proof}
\begin{corollary}
We  have
\begin{equation}\label{XXX}
\sum_{k = 1}^n \frac{C_k C_{n - k}}{k} = \left(2H_{2n} - 2H_{n + 1} + 3 \right)C_n - C_{n+1}.
\end{equation}
\end{corollary}
\begin{proof} 
Write \eqref{pzfs47k} as 
\begin{gather*}
\sum_{k = 1}^n C_k C_{n - k} + \sum_{k = 1}^n \frac{C_k C_{n - k}}{k} = 2C_n (H_{2n} - H_{n + 1} + 1 ).
\end{gather*}

Using
$ \sum\limits_{k=0}^n C_kC_{n-k}=C_{n+1}, 
$ we obtain \[ C_{n+1}-C_n + \sum_{k=1}^n\frac{C_kC_{n-k}}{k} = 2C_n(H_{2n}-H_{n+1}+1), \] 
and therefore \eqref{XXX}.
\end{proof}
\begin{corollary}
	We have
	\begin{equation}\label{H_odd_xy}
	\sum_{k=1}^n \frac{2k+1}{2^{2k}} (O_{2k} - O_{k}) C_{2k}C_{n-k} =  \frac{C_{2n}}{2^{2n}} \Big ( (4n+1)(O_{2n} - O_{n}) - \frac{2n}{2n+1} \Big)
	\end{equation}
	and
	\begin{align}\label{H_odd_yz}
	&\sum_{k=1}^n \frac{4k+1}{2^{2k}} (O_{2k+1} - O_{k+1}) C_{2k}C_{n-k} \nonumber \\
	&\qquad\qquad\qquad\qquad =  \frac{4n+1}{2n+3}\, \frac{C_{2n}}{{2^{2n}}}\Big( (4n+3) (O_{2n+1} - O_{n+1}) - \frac{2n}{2n+3} \Big).
	\end{align}
\end{corollary}
\begin{proof}
	To get \eqref{H_odd_xy} set $x=-\frac12$ in \eqref{CH_bin_id2}, use
	$$H_{2k-\frac12} - H_{k-\frac12} = 2(O_{2k} - O_{k}),$$
	and simplify. Equation \eqref{H_odd_yz} is obtained from \eqref{CH_bin_id2} by inserting $x=\frac12$ and making use of
	$$H_{2k+\frac12} - H_{k+\frac12} = 2\,(O_{2k+1} - O_{k+1}).$$
\end{proof}

\section{Conclusion}

We have developed a systematic method for deriving finite and infinite identities involving Catalan numbers from generalized binomial identities. The main ingredient is the representation of binomial coefficients with half-integer parameters in terms of Catalan numbers and products of consecutive odd integers. Combined with Pascal-type identities, this representation yields a variety of closed-form evaluations for Catalan-number series and related finite sums.

We have also shown that differentiation with respect to a complex parameter provides a natural way to generate harmonic-number extensions of these identities. In particular, we obtained families of identities involving harmonic and odd harmonic numbers, including convolution formulas and finite and infinite series. The use of identities of Bat{\i}r and Sofo further extends the method to sums involving higher-order harmonic quantities.

The generating-function approach developed in the final part of the paper provides additional convolution identities involving Catalan numbers and generalized binomial coefficients, together with their harmonic analogues. Several known identities arise as special cases, while the general formulas provide a common framework for a number of related summation results.

The methods presented here can be applied to other binomial identities with suitable parameter dependence and may therefore provide further families of Catalan-number and harmonic-number identities.

\end{document}